\documentclass[11pt,twoside]{article}
\usepackage[hmargin=3cm,vmargin=1in,a4paper,centering]{geometry}
\usepackage[T1]{fontenc}
\usepackage[lining]{ebgaramond}
\usepackage{amsmath,amsthm}
\usepackage[ebgaramond]{newtxmath}
\usepackage{mleftright}
\mleftright
\usepackage[nobottomtitles]{titlesec}
\titleformat{\section}{\normalfont\large\bfseries}{\filcenter\S\thesection.}{1ex}{\filcenter}
\titleformat{\subsection}{\normalfont\bfseries}{\filcenter\S\thesubsection.}{1ex}{\filcenter}
\usepackage[colorlinks,allcolors=blue,pagebackref]{hyperref}
\usepackage{pifont}
\renewcommand*{\backrefalt}[4]{%
\ifcase #1 %
\textcolor{red}{No citations}%
\or
\ding{43}~p.~#2%
\else
\ding{43}~pp.~#2%
\fi}
\usepackage{breakcites}
\allowdisplaybreaks
\ifdefined\iflatexml
\theoremstyle{plain}
\newtheorem{theorem}{Theorem}

\newtheorem{lemma}[theorem]{Lemma}
\newtheorem{corollary}[theorem]{Corollary}
\theoremstyle{definition}

\newtheorem{remark}[theorem]{Remark}
\else
\usepackage{tcolorbox,varwidth}
\tcbuselibrary{skins, breakable, theorems}
\colorlet{thmcolor}{orange!80!white}
\colorlet{remcolor}{teal!50!white}
\colorlet{defncolor}{blue!50!black}
\usepackage{keytheorems}
\tcbset{commonstyle/.style={%
colback=white!95!tcbcolframe,
colbacktitle=white!80!tcbcolframe,
arc=2mm,
fonttitle=\bfseries,
coltitle=black,
enhanced,
varwidth boxed title*=-2cm,
attach boxed title to top left={yshift=-3mm,xshift=0.5cm,yshifttext=-1mm},
description font=\mdseries,
breakable,
before upper={\parindent15pt\noindent},
beforeafter skip balanced=10.0pt plus 1.0pt minus 1.0pt,
separator sign={\textbf{: }},
overlay first={\draw[color=tcbcolframe, line width=.5pt] (frame.south west)--(frame.south east);},
overlay middle={\draw[color=tcbcolframe, line width=.5pt] (frame.south west)--(frame.south east);\draw[color=tcbcolframe, line width=.5pt] (frame.north west)--(frame.north east);},
overlay last={\draw[color=tcbcolframe, line width=.5pt] (frame.north west)--(frame.north east);}
}
}
\newkeytheoremstyle{mythmstyle}{
bodyfont=\normalfont,
headpunct={},
notebraces={}{},
noteseparator={: },
notefont=\bfseries,
tcolorbox = {colframe=thmcolor,commonstyle}
}
\newkeytheoremstyle{myremstyle}{
bodyfont=\normalfont,
headpunct={},
notebraces={}{},
noteseparator={: },
notefont=\bfseries,
tcolorbox = {colframe=remcolor,commonstyle}
}
\newkeytheoremstyle{mydefstyle}{
bodyfont=\normalfont,
headpunct={},
notebraces={}{},
noteseparator={: },
notefont=\bfseries,
tcolorbox = {colframe=defncolor,commonstyle}
}
\newkeytheorem{theorem}[
name=Theorem,
style=mythmstyle
]
\newkeytheorem{lemma}[
name=Lemma,
style=mythmstyle,
sibling=theorem
]
\newkeytheorem{proposition}[
name=Proposition,
style=mythmstyle,
sibling=theorem
]
\newkeytheorem{corollary}[
name=Corollary,
style=mythmstyle,
sibling=theorem
]
\newkeytheorem{remark}[
name=Remark,
sibling=theorem,
style=myremstyle
]
\newkeytheorem{definition}[
name=Definition,
style=mydefstyle,
sibling=theorem
]
\fi
\newcommand{\R}{\mathbb R}
\newcommand{\C}{\mathbb C}
\newcommand{\Z}{\mathbb Z}
\newcommand{\Om}{\Omega}

\newcommand{\cP}{\mathcal P}
\newcommand{\cA}{\mathcal A}
\newcommand{\cB}{\mathcal B}
\newcommand{\cE}{\mathcal E}
\newcommand{\lamP}{\lambda^{\mathrm P}}
\newcommand{\Lcln}[2]{L^{\mathrm{cl}}_{#1,#2}}
\newcommand{\ir}{\mathrm i}
\newcommand{\dr}{\mathrm d}
\newcommand{\er}{\mathrm e}
\newcommand{\mydoi}[1]{\href{https://doi.org/#1}{doi: #1}}
\newcommand{\myarXiv}[1]{\href{https://arxiv.org/abs/#1}{arXiv: #1}}
\DeclareMathOperator{\Tr}{Tr}
\DeclareMathOperator{\curl}{curl}
\renewcommand{\Re}{\operatorname{Re}}
\renewcommand{\tilde}{\widetilde}
\renewcommand{\bar}{\overline}
\renewcommand\footnotemark{}
\begin{document}
\title{Berezin--Li--Yau inequalities\\for Aharonov--Bohm operators}
\author{Michael Levitin\thanks{%
\textbf{M. L.: }Department of Mathematics and Statistics, University of Reading,
Pepper Lane, Whiteknights, Reading RG6 6AX, UK;
\href{mailto:M.Levitin@reading.ac.uk}{\nolinkurl{M.Levitin@reading.ac.uk}}; \url{https://www.michaellevitin.net}; ORCID: 0000-0003-0020-3265%
}
\and
Iosif Polterovich\thanks{%
\textbf{I. P.: }D\'e\-par\-te\-ment de math\'ematiques et de statistique, Univer\-sit\'e de Mont\-r\'eal,
CP 6128 succ Centre-Ville, Mont\-r\'eal QC  H3C 3J7, Canada;
\href{mailto:iosif.polterovich@umontreal.ca}{\nolinkurl{iosif.polterovich@umontreal.ca}}; \url{https://www.dms.umontreal.ca/\~iossif}; ORCID: 0009-0007-0052-6589%
}
\and
Romain Speciel\thanks{%
\textbf{R. S.: }Department of Mathematics, Stanford University,
Building 380, Stanford CA 94305, USA;
\href{mailto:rspeciel@stanford.edu}{\nolinkurl{rspeciel@stanford.edu}}; \url{https://web.stanford.edu/~rspeciel/}; ORCID: 0009-0009-0032-3145%
}}

\date{\small 29 October 2026}
\maketitle

\begin{abstract}
We prove the sharp Berezin--Li--Yau inequality for  Riesz means of order at least one for Aharonov--Bohm operators on  bounded planar domains with Dirichlet boundary conditions. This confirms a conjecture posed by Frank and Hansson in 2008. We also obtain an improved lower bound for individual eigenvalues and extend it to the case of magnetic potentials with bounded magnetic field. These results are then applied to deduce estimates for Riesz means  of orders between zero and one, and for Aharonov--Bohm eigenvalues on product domains.
\end{abstract}

\begingroup\small\noindent
\par\endgroup

\section{Introduction and main results}\label{sec:intro}

Let $\Om\subset\R^2$ be a nonempty bounded open set with Dirichlet Laplacian $-\Delta_\Om$, whose $j$th eigenvalue we denote by $\lambda_j(-\Delta_\Om)$. The Berezin--Li--Yau inequality \cite{Berezin,LiYau} states
\begin{equation}\label{eq:BLY}
\sum_{j=1}^k\lambda_j(-\Delta_\Om)\geq\frac{2\pi}{|\Om|}k^2,
\qquad k\geq1.
\end{equation}
More generally, for a nonnegative self-adjoint operator $H$ with a discrete spectrum
\[
0\le\lambda_1=\lambda_1(H)\leq\lambda_2=\lambda_2(H)\leq\cdots,
\]
we may consider its \emph{Riesz mean} of order $\gamma\ge 0$, defined for $\Lambda\in\mathbb{R}$ by
\[
\Tr(\Lambda-H)_+^\gamma
=\sum_j(\Lambda-\lambda_j(H))_+^\gamma,
\]
where $s_+:=\max\{s,0\}$.

Note that the $\gamma=0$ case simply expresses the \emph{eigenvalue counting function}
\[
\mathcal{N}_H(\Lambda)=\#\left\{j: \lambda_j(H)<\Lambda\right\}
\]
of $H$. In the case of $H=-\Delta_\Omega$, the $\gamma=1$ case corresponds to \eqref{eq:BLY}, which is seen to be equivalent to
\[
\Tr(\Lambda+\Delta_\Omega)_+
=\sum_j(\Lambda-\lambda_j(-\Delta_\Omega))_+
\leq \frac{|\Omega|}{8\pi}\Lambda^2,
\]
as explained in \cite{LW}, \cite[Chapter~3]{FLW23}.

We study Riesz means of \emph{magnetic Laplace operators}, which we now introduce. Let $\cP\subset\Om$ be a finite set and $\cA$ be a continuously differentiable real vector potential on $\Om\setminus\cP$. Define $H_\cA=(-\ir\nabla-\cA)^2$,  the magnetic Laplace operator corresponding to $\cA$ with associated quadratic form
\begin{equation}\label{eq:form}
\mathfrak h[u]=\mathfrak h_\cA[u]:=\int_\Om|(-\ir\nabla-\cA)u|^2\,\dr x,
\qquad u\in C_c^\infty(\Om\setminus\cP).
\end{equation}
The operator $H_\cA$ is the self-adjoint operator associated with the closure of this form in $L^2(\Om)$. It is positive, has a compact resolvent, and a discrete spectrum of positive eigenvalues. Note that when $\cP=\varnothing$ and $\cA=0$, it is simply the Dirichlet Laplacian.

We say that the Riesz mean of $H_\cA$ satisfies the \emph{semiclassical bound} of order $\gamma\geq0$ if
\begin{equation}\label{eq:semiclassical}
\Tr(\Lambda-H_\cA)_+^\gamma\leq\frac{|\Om|}{4\pi(\gamma+1)}\Lambda^{\gamma+1}
\qquad\text{ for all }\Lambda\geq0.
\end{equation}
Weyl's law \cite[Theorem~A.1]{Frank09} gives
\[
\Tr(\Lambda-H_\cA)_+^\gamma = \frac{|\Om|}{4\pi(\gamma+1)}\Lambda^{\gamma+1}+o\left(\Lambda^{\gamma+1}\right)\qquad\text{as }\Lambda\to\infty, 
\]
hence the constant in \eqref{eq:semiclassical} cannot be improved. For the Dirichlet Laplacian, Berezin's theorem \cite{Berezin} asserts that \eqref{eq:semiclassical} holds for all $\gamma\geq1$. For general magnetic fields, the semiclassical bound for Riesz means of order $\gamma\geq3/2$ was obtained in \cite{LW}, while for constant fields it was proved for every order $\gamma\geq1$ in  \cite{ELV}.

An important example of the magnetic Laplace operator is the one with the \emph{Aharonov--Bohm potential}. Writing $\cA=(\cA_1,\cA_2)$, the \emph{magnetic field} is given by $\cB=\curl\cA=\partial_1\cA_2-\partial_2\cA_1$ on $\Om\setminus\cP$, and we call $\cA$ \emph{curl-free} if $\cB=0$ there. The Aharonov--Bohm potential
\[
\cA_{p,\alpha}(x)=\alpha\frac{(x-p)^\perp}{|x-p|^2},
\qquad (x_1,x_2)^\perp=(-x_2,x_1),
\]
with \emph{pole} $p\in\Om$ and \emph{flux} $\alpha\in\R$, is curl-free on $\Om\setminus\{p\}$. We denote the associated magnetic Laplace operator, called the \emph{Aharonov--Bohm operator}, by $H_{p,\alpha}$. Although its magnetic field vanishes away from $p$, for $\alpha\notin\Z$ the operator $H_{p,\alpha}$ is not gauge equivalent to the Dirichlet Laplacian, and its spectrum depends on $\alpha$.

In this paper, we investigate Riesz means of Aharonov--Bohm operators. It was proved in \cite{FH} that the order-one semiclassical bound \eqref{eq:ABorderone} for $H_{p,\alpha}$ holds, with the right-hand side multiplied by a constant $c>1$ independent of $\alpha$. The numerical evaluation therein gave $c\approx1.0224$. The authors then conjectured that the bound holds with $c=1$. We prove this conjecture as a consequence of the following theorem, which is our main result. To state it, set
\[
m_k(H)=\frac{1}{k}\sum_{j=1}^k\lambda_j(H)
\]
and suppress $H$ from the notation when it is understood in context.

\begin{theorem}\label{thm:ABmain}
Let $\Om\subset\R^2$ be a nonempty bounded open set, $\cP\subset\Om$ a finite set, and $\cA$ a real smooth curl-free potential on $\Om\setminus\cP$. Then the eigenvalues of $H_\cA$ satisfy $4m_k>\lambda_k$ and
\begin{equation}\label{eq:ABindividual}
\lambda_k\geq\frac{2m_k}{4m_k-\lambda_k}\,\frac{4\pi k}{|\Om|},
\qquad k\geq1.
\end{equation}
In particular, \eqref{eq:ABindividual} holds for the Dirichlet Aharonov--Bohm operator $H_{p,\alpha}$ for every $p\in\Om$ and every $\alpha\in\R$.
\end{theorem}

The case when $\cA=0$, which corresponds simply to the Dirichlet Laplacian, was recently obtained in \cite{Sp} by applying a trace identity for an auxiliary multiplication operator. Our proof of Theorem~\ref{thm:ABmain} extends this method, following \cite[Remark~6]{Sp}. We apply the Berger--Shaw theorem \cite{BS} to give a bound on the trace of a certain commutator in terms of the area of the domain, and the sum rules of \cite{LP} then yield the result.

Theorem \ref{thm:ABmain} has several immediate corollaries. First, we obtain a Berezin--Li--Yau inequality in this setting.

\begin{corollary}\label{cor:ABsum}
Under the assumptions of Theorem~\ref{thm:ABmain},
\begin{equation}\label{eq:sharpABsum}
\sum_{j=1}^k\lambda_j\geq\frac{2\pi}{|\Om|}k^2,
\qquad k\geq1.
\end{equation}
The constant in this inequality cannot be improved.
\end{corollary}

\begin{proof}[Proof of Corollary \ref{cor:ABsum}]
Multiplying \eqref{eq:ABindividual} by its positive denominator gives
\[
\frac{8\pi k\,m_k}{|\Om|}\leq\lambda_k(4m_k-\lambda_k)
=4m_k^2-(\lambda_k-2m_k)^2\leq4m_k^2.
\]
Thus $m_k\geq2\pi k/|\Om|$, which is \eqref{eq:sharpABsum}. That the constant cannot be improved follows from Weyl's law already for $\cA=0$.
\end{proof}

As mentioned above, the eigenvalue-sum inequality \eqref{eq:sharpABsum} is equivalent to the order-one Riesz mean bound
\begin{equation}\label{eq:ABorderone}
\Tr\left(\Lambda-H_\cA\right)_+
\leq\frac{|\Om|}{8\pi}\Lambda^2,
\qquad\Lambda\geq0.
\end{equation}
Integrating this bound in $\Lambda$ yields the semiclassical bound for Riesz means of every order $\gamma\geq1$, which for $H_{p,\alpha}$ is the Frank--Hansson conjecture.

\begin{corollary}[The Frank--Hansson conjecture]\label{cor:ABriesz}
Under the assumptions of Theorem~\ref{thm:ABmain},
\[
\Tr\left(\Lambda-H_\cA\right)_+^\gamma
\leq\frac{|\Om|}{4\pi(\gamma+1)}\Lambda^{\gamma+1},
\qquad\gamma\geq1,\quad\Lambda\geq0.
\]
The constant cannot be improved.
\end{corollary}

\begin{proof}[Proof of Corollary \ref{cor:ABriesz}]
The case $\gamma=1$ is \eqref{eq:ABorderone}. For $\gamma>1$, use
\[
(\Lambda-\lambda)_+^\gamma
=\gamma(\gamma-1)\int_0^\Lambda
(\Lambda-t)^{\gamma-2}(t-\lambda)_+\,\dr t.
\]
Summing over the eigenvalues and applying \eqref{eq:ABorderone} gives
\[
\Tr\left(\Lambda-H_\cA\right)_+^\gamma
\leq\frac{|\Om|\gamma(\gamma-1)}{8\pi}
\int_0^\Lambda(\Lambda-t)^{\gamma-2}t^2\,\dr t
=\frac{|\Om|}{4\pi(\gamma+1)}\Lambda^{\gamma+1}.
\]
Optimality of the constant follows already from Weyl's law with $\cA=0$.
\end{proof}

We also deduce a P\'olya-type lower bound for individual eigenvalues.

\begin{corollary}\label{cor:ABindividual}
Under the assumptions of Theorem~\ref{thm:ABmain},
\begin{equation}\label{eq:twothirds}
\lambda_k\geq\frac23\,\frac{4\pi k}{|\Om|},
\qquad k\geq1.
\end{equation}
\end{corollary}

\begin{proof}[Proof of Corollary \ref{cor:ABindividual}]
Since $m_k\leq\lambda_k$, we have $2m_k/(4m_k-\lambda_k)\geq2/3$. Substituting this into \eqref{eq:ABindividual} proves \eqref{eq:twothirds}.
\end{proof}

The individual bound \eqref{eq:twothirds} improves the bound $\lambda_k\geq 2\pi k/|\Om|$, which follows from \eqref{eq:sharpABsum} alone since $\lambda_k\geq m_k$. For the disc with the pole at its centre, P\'olya's inequality $\lambda_k\geq 4\pi k/|\Om|$ was proved in \cite{FLPS}. The question whether the coefficient $2/3$ in \eqref{eq:twothirds} can be improved for arbitrary Aharonov--Bohm configurations remains open.

\begin{remark}\label{rem:gauge}
Theorem~\ref{thm:ABmain} imposes no condition on $\cA$ near the poles beyond the choice of $C_c^\infty(\Om\setminus\cP)$ as the domain of definition for the quadratic form in \eqref{eq:form}. In particular, one may take $\cA=\sum_{p\in\cP}\cA_{p,\alpha_p}$ with arbitrary real fluxes $\alpha_p$, and one may add to any admissible $\cA$ a smooth curl-free potential on $\Om$, which need not be a gradient when $\Om$ has holes. Conversely, near each pole $p$ a curl-free potential is, up to a unitary gauge transformation $u\mapsto\er^{\ir\varphi_p}u$, the Aharonov--Bohm potential $\cA_{p,\alpha_p}$ with the flux $\alpha_p=1/(2\pi)\oint\cA$; poles with integer $\alpha_p$ are removable.
\end{remark}

In Section~\ref{sec:elementary}, we give an alternative proof of the area estimate \eqref{eq:area_bound}, a key step in the proof of Theorem~\ref{thm:ABmain}. As a consequence, we obtain the following theorem.

\begin{theorem}\label{thm:magnetic}
Let $\Om\subset\R^2$ be a nonempty bounded open set, $\cP\subset\Om$ a finite set, and $\cA$ a real $C^1$ potential on $\Om\setminus\cP$ with $\cB=\curl\cA\in L^\infty(\Om\setminus\cP)$. Write $b:=\|\cB\|_\infty$. Then the eigenvalues of $H_\cA$ satisfy $4m_k-\lambda_k+b>0$ and
\begin{equation}\label{eq:magrefined}
2k\left(\frac1{\lambda_k+b}
+\frac1{4m_k-\lambda_k+b}\right)\leq\frac{|\Om|}{\pi},
\qquad k\geq1.
\end{equation}
Consequently,
\begin{equation}\label{eq:magindividual}
\lambda_k\geq\frac13\left(\frac{4\pi k}{|\Om|}-2b+\sqrt{\left(\frac{4\pi k}{|\Om|}\right)^2-\frac{4\pi kb}{|\Om|}+b^2}\right)
\geq\frac23\,\frac{4\pi k}{|\Om|}-b,
\qquad k\geq1.
\end{equation}
\end{theorem}

Note that $b$ bounds only the regular part of the field, while the fluxes at the poles are unrestricted. For $b=0$, \eqref{eq:magrefined} reduces to \eqref{eq:ABindividual} and \eqref{eq:magindividual} to \eqref{eq:twothirds}. See Remark~\ref{rem:admissible} for a discussion of the singularities of $\cA$ admitted by Theorem~\ref{thm:magnetic}.

The paper is organised as follows. Section~\ref{sec:AB} gives the Berger--Shaw proof of Theorem~\ref{thm:ABmain}. Section~\ref{sec:elementary} presents an alternative derivation of the area estimate used in that argument, which is then extended in Section~\ref{sec:magnetic} in order  to prove Theorem~\ref{thm:magnetic}. Finally, Section~\ref{sec:further} collects two further consequences of Theorem~\ref{thm:ABmain}: semiclassical and P\'olya-type bounds for Aharonov--Bohm operators on product domains $D\times G$ with $D\subset\R^2$ (Theorem~\ref{thm:products}), and bounds for Riesz means of order $0<\gamma<1$, with an application to cylinders (Theorem~\ref{thm:fractional}).

\section{The Berger--Shaw proof}\label{sec:AB}

We present here a proof of Theorem~\ref{thm:ABmain} based on a trace inequality of Berger and Shaw \cite[Theorem~1]{BS}. Throughout this section, $\Om$, $\cP$ and $\cA$ are as in Theorem~\ref{thm:ABmain}. We shall assume for now that $\Omega$ has smooth boundary and that $\mathcal A$ is smooth on a neighbourhood of $\overline\Omega$. We return to the general case by a standard argument at the end of the proof.

Put $P_r=-\ir\partial_r-\cA_r$ for $r=1,2$, so that $H_\cA$ is the Dirichlet realisation of $P_1^2+P_2^2$. The form domain of $H_\cA$ is $\cE=H_0^1(\Om;\C)$, which we equip with the energy inner product
\[
\langle u,v\rangle_{\cE}
=\mathfrak h[u,v]
=\langle P_1u,P_1v\rangle+\langle P_2u,P_2v\rangle.
\]
For $f\in C^\infty(\overline\Om;\C)$, let $M_fu=fu$ on $\cE$, with adjoint $M_f^*$ taken in this space. We identify $\R^2$ with $\C$ by $z=x_1+\ir x_2$, and define the closed subspaces
\[
V=\{v\in L^2(\Om):(P_1+\ir P_2)v=0\text{ in the sense of distributions on }\Om\},
\qquad W=V^\perp.
\]
Since $(P_1+\ir P_2)(zv)=z(P_1+\ir P_2)v$, multiplication by $z$ preserves $V$. On these spaces, we consider the bounded operators
\[
Z:V\to V,\qquad Zv=zv,
\qquad K:W\to V,\qquad Kw=\operatorname{Proj}_V(zw),
\]
with adjoints taken in the inherited $L^2$ inner products.

A bounded operator $T$ on a Hilbert space is said to be \emph{hyponormal} if $[T^*,T]\geq0$, and \emph{rationally cyclic} if there is a vector $v$ such that every vector in the space can be approximated in norm by vectors $f(T)v$, where $f$ is rational with no poles in $\operatorname{Spec}(T)$. We shall use the following special case of the Berger--Shaw theorem.

\begin{theorem}[{\cite[Theorem~1]{BS}}]\label{thm:BS}
If $T$ is a bounded rationally cyclic hyponormal operator, then $[T^*,T]$ is trace class and
\[
\Tr[T^*,T]\leq\frac{|\operatorname{Spec}(T)|}{\pi},
\]
where $|\operatorname{Spec}(T)|$ is the Lebesgue measure of the spectrum $\operatorname{Spec}(T)\subset\C$.
\end{theorem}

Recall that the Bergman space $A^2(\Om)$ is the closed subspace of $L^2(\Om)$ consisting of holomorphic functions.

\begin{lemma}\label{lem:area}
The operator $C=[M_{x_1}^*,M_{x_1}]+[M_{x_2}^*,M_{x_2}]$ on $\cE$ is nonnegative and trace class, with
\begin{equation}\label{eq:ABareatrace}
\Tr C\leq\frac{|\Om|}{\pi}.
\end{equation}
\end{lemma}

\begin{proof}
Set $Q=P_1-\ir P_2$. Since $\curl\cA=0$, integration by parts gives
\[
\|Qu\|^2=\|u\|_{\cE}^2=\mathfrak h[u], \qquad Q(\bar zu)=\bar z\,Qu,
\qquad u\in\cE.
\]
The first identity shows that $Q:\cE\to L^2(\Om)$ is an isometry with closed range and, since $(Q\cE)^\perp=\ker(P_1+\ir P_2)=V$, that $Q$ maps $\cE$ unitarily onto $W$. The second identity shows that multiplication by $\bar z$ preserves $W$ and that $Q$ intertwines $M_{\bar z}$ on $\cE$ with multiplication by $\bar z$ on $W$.

With $\partial_{\bar z}=(\partial_1+\ir\partial_2)/2$, choose $F$ smooth on a neighbourhood of $\overline\Om$ such that
\[
\partial_{\bar z}F=\frac{\ir}{2}(\cA_1+\ir\cA_2).
\]
Such an $F$ exists by the solvability of the inhomogeneous Cauchy--Riemann equation \cite[Theorem~1.4.4]{Hormander}. Since
\[
P_1+\ir P_2=-2\ir \er^F\partial_{\bar z}\er^{-F}
\]
and $\er^{\pm F}$ are bounded on $\Om$, the kernel $V$ of $P_1+\ir P_2$ consists of the functions $\er^Fg$ with $g$ holomorphic and square-integrable, that is,
\[
V=\er^F A^2(\Om).
\]
Relative to $L^2(\Om)=V\oplus W$, multiplication by $z$ has the block matrix
\[
\begin{pmatrix}
Z & K\\
0 & Q M_{\bar z}^*Q^{-1}
\end{pmatrix},
\]
where the lower-right block is the adjoint of multiplication by $\bar z$ on $W$. Comparing the diagonal blocks in the normality identity for multiplication by $z$ on $L^2(\Om)$ gives
\begin{equation}\label{eq:normalblocks}
[Z^*,Z]=KK^*,
\qquad Q[M_{\bar z}^*,M_{\bar z}]Q^{-1}=K^*K.
\end{equation}
In particular, both self-commutators are nonnegative, so $Z$ is hyponormal. Furthermore, since the spectrum of $Z$ lies in $\overline\Om$, and rational functions with poles off $\overline\Om$ are dense in $A^2(\Om)$ by \cite[p.~162, Corollary to Theorems~4 and~5]{Hedberg},  $Z$ is rationally cyclic with cyclic vector $\er^F$.

Theorem~\ref{thm:BS} implies that $KK^*=[Z^*,Z]$ is trace class, hence $K^*K$ is also trace class and has the same trace. Using \eqref{eq:normalblocks}, the unitarity of $Q$, and $|\overline\Om|=|\Om|$ (as $\partial\Om$ is smooth), we obtain
\[
\Tr[M_{\bar z}^*,M_{\bar z}]
=\Tr(K^*K)=\Tr(KK^*)
=\Tr[Z^*,Z]
\leq\frac{|\operatorname{Spec}(Z)|}{\pi}
\leq\frac{|\overline\Om|}{\pi}=\frac{|\Om|}{\pi}.
\]

Finally, note that complex conjugation maps the energy space for $\cA$ antiunitarily onto that for $-\cA$ and exchanges $M_z$ and $M_{\bar z}$. We may therefore apply the same argument to $-\cA$ to obtain the corresponding assertions for $[M_z^*,M_z]$, and the lemma follows since
\[
	2C=[M_{z}^*,M_{z}]+[M_{\bar z}^*,M_{\bar z}].
\]
\end{proof}

\begin{proof}[Proof of Theorem~\ref{thm:ABmain}]
Multiplication by either coordinate preserves the domain of $H_\cA$, and
\[
[H_\cA,x_r]=-2\ir P_r,
\qquad [[H_\cA,x_r],x_r]=-2I,
\qquad r=1,2.
\]
Now, let $(u_i)$ be an orthonormal eigenbasis of $H_\cA$, and set
\[
\rho_{ij}=\sum_{r=1}^2|\langle x_ru_i,u_j\rangle|^2=\rho_{ji},
\qquad
d_{ij}=(\lambda_j-\lambda_i)\rho_{ij}=-d_{ji}.
\]
Since $\|P_1u_i\|^2+\|P_2u_i\|^2=\lambda_i$, the trace identities of \cite[Theorem~2.2, (2.3) and (2.5)]{LP} give
\begin{equation}\label{eq:ABmoments}
\sum_jd_{ij}=2,
\qquad \sum_j(\lambda_i+\lambda_j)d_{ij}=8\lambda_i.
\end{equation}
The vectors $e_i=u_i/\sqrt{\lambda_i}$ form an orthonormal basis of $\cE$, in which
\[
\langle M_{x_r}e_i,e_j\rangle_{\cE}
=\sqrt{\frac{\lambda_j}{\lambda_i}}\,
\langle x_ru_i,u_j\rangle,
\qquad r=1,2.
\]
Expanding $\|M_{x_r}e_i\|_\cE^2-\|M_{x_r}^*e_i\|_\cE^2$ in this basis and summing over $r$, we obtain
\begin{equation}\label{eq:diagonal}
\langle Ce_i,e_i\rangle_{\cE}
=\sum_j\left(\frac{\lambda_j}{\lambda_i}-\frac{\lambda_i}{\lambda_j}\right)\rho_{ij}
=\sum_j\left(\frac1{\lambda_i}+\frac1{\lambda_j}\right)d_{ij}.
\end{equation}

Fix $k\geq1$. We make explicit the weighted-mean argument of \cite[Lemma~5]{Sp}, obtaining a slightly stronger conclusion. Summing \eqref{eq:ABmoments} over $i\leq k$ and using $d_{ij}=-d_{ji}$ to cancel the terms with $i,j\leq k$, we obtain
\[
\sum_{i\leq k<j}d_{ij}=2k,
\qquad
\sum_{i\leq k<j}(\lambda_i+\lambda_j)d_{ij}
=8km_k.
\]
Since $d_{ij}\geq0$ for $i\leq k<j$, the weighted means
\[
\mu_-=\frac1{2k}\sum_{i\leq k<j}\lambda_id_{ij},
\qquad
\mu_+=\frac1{2k}\sum_{i\leq k<j}\lambda_jd_{ij}
\]
satisfy
\[
\mu_-+\mu_+=4m_k,
\qquad
0<\mu_-\leq\lambda_k\leq\lambda_{k+1}\leq\mu_+.
\]
Similarly, summing \eqref{eq:diagonal} over $i\le k$ and cancelling the terms with $i,j\leq k$, then applying the Cauchy--Schwarz inequality separately to the two reciprocal sums, gives

\[
\sum_{i=1}^k\langle Ce_i,e_i\rangle_{\cE}
=
\sum_{i\leq k<j}
\left(\frac1{\lambda_i}+\frac1{\lambda_j}\right)d_{ij}
\geq
2k\left(\frac1{\mu_-}+\frac1{\mu_+}\right).
\]
Now let $\Lambda\in[\lambda_k,\lambda_{k+1}]$. Then $\mu_-\leq\Lambda\leq\mu_+$, so
\[
4m_k-\Lambda=\mu_-+\mu_+-\Lambda\geq\mu_->0
\]
and
\[
\Lambda(4m_k-\Lambda)-\mu_-\mu_+
=(\Lambda-\mu_-)(\mu_+-\Lambda)\geq0.
\]
Consequently,
\[
\frac1{\mu_-}+\frac1{\mu_+}
=\frac{4m_k}{\mu_-\mu_+}
\geq\frac{4m_k}{\Lambda(4m_k-\Lambda)}
=\frac1\Lambda+\frac1{4m_k-\Lambda}.
\]
Combining these inequalities with Lemma~\ref{lem:area} yields, for every $\Lambda\in[\lambda_k,\lambda_{k+1}]$,
\begin{equation}\label{eq:ABrefined}
2k\left(\frac1\Lambda+\frac1{4m_k-\Lambda}\right)
\leq\sum_{i=1}^k\langle Ce_i,e_i\rangle_{\cE}
\leq\Tr C\leq\frac{|\Om|}{\pi}.
\end{equation}
The inequality between the outer terms of \eqref{eq:ABrefined} can be rewritten as
\begin{equation}\label{eq:ABrefinedscalar}
\Lambda(4m_k-\Lambda)\geq\frac{8\pi k\,m_k}{|\Om|},
\qquad\Lambda\in[\lambda_k,\lambda_{k+1}].
\end{equation}

Finally, as mentioned above, we address the case of general domains that may not have a smooth boundary. Let $(\Om_n)$ be an increasing exhaustion of $\Om\setminus\cP$ by relatively compact open sets with smooth boundary, and let $\lambda_{j,n}$ be the $j$th eigenvalue of the magnetic Laplace operator on $\Om_n$ with potential $\cA$. Since every function in $C_c^\infty(\Omega\setminus\cP)$ is supported in some $\Omega_n$, the definition of $H_\cA$ and the min–max principle give $\lambda_{j,n}\downarrow\lambda_j$ for every $j$, while $|\Om_n|\to|\Om|$. For any $\Lambda\in[\lambda_k,\lambda_{k+1}]$, apply \eqref{eq:ABrefinedscalar} on $\Om_n$ with $\Lambda_n=\max\{\Lambda,\lambda_{k,n}\}$ and let $n\to\infty$. Thus \eqref{eq:ABrefinedscalar} holds on the original domain. Taking $\Lambda=\lambda_k$ gives
\[
\lambda_k(4m_k-\lambda_k)\geq\frac{8\pi k\,m_k}{|\Om|}>0.
\]
In particular $4m_k-\lambda_k>0$, and dividing by $4m_k-\lambda_k$ gives \eqref{eq:ABindividual}.
\end{proof}

\begin{remark}\label{rem:areaequality}
The bound \eqref{eq:ABareatrace} is in fact an equality. Indeed, \eqref{eq:ABrefined} and $s^{-1}+t^{-1}\geq4/(s+t)$ for $s,t>0$ give $\Tr C\geq2k/m_k$. Weyl's law gives $m_k\sim2\pi k/|\Om|$, so letting $k\to\infty$ yields $\Tr C=|\Om|/\pi$.
\end{remark}

\section{An alternative proof of the area bound}\label{sec:elementary}

We give an alternative proof of the area bound
\begin{equation}\label{eq:area_bound}
\sum_{i\leq k}\langle Ce_i,e_i\rangle_{\cE}\leq\frac{|\Om|}{\pi},
\end{equation}
which supplies the upper bound needed in \eqref{eq:ABrefined} and is all that the proof of Theorem \ref{thm:ABmain} requires. 
This proof also uses the commutator identities,  but instead of the Berger--Shaw theorem it relies on the elementary  identity \eqref{eq:flatHodge} and Weyl's law .

The notation and smoothness assumptions are those of Section~\ref{sec:AB}.

\begin{lemma}\label{lem:flatdiv}
For $w_1,w_2\in\cE$,
\begin{equation}\label{eq:flatHodge}
\|P_1w_1+P_2w_2\|^2+\|P_1w_2-P_2w_1\|^2=\sum_{r,s=1}^2\|P_rw_s\|^2 .
\end{equation}
Consequently, for an $L^2$-normalised eigenfunction $u_i$,
\begin{equation}\label{eq:flatq}
q_i:=\sum_{r=1}^2\left\|H_\cA^{-1/2}P_ru_i\right\|^2\leq1 .
\end{equation}
\end{lemma}
\begin{proof}
For $w_1,w_2\in C_c^\infty(\Om)$, expand the two squared norms on the left-hand side of \eqref{eq:flatHodge}. The cross terms are then
\[
2\Re\left(\langle P_1w_1,P_2w_2\rangle-\langle P_1w_2,P_2w_1\rangle\right).
\]

Now,
\[
\langle P_1w_1,P_2w_2\rangle=\langle P_2P_1w_1,w_2\rangle=\langle P_1P_2w_1,w_2\rangle=\langle P_2w_1,P_1w_2\rangle
\]
by $[P_1,P_2]=\ir\curl\cA=0$ and the symmetry of the $P_r$ on $C_c^\infty(\Om)$, so the cross terms cancel. Density gives \eqref{eq:flatHodge} on $\cE$.

In order to prove \eqref{eq:flatq}, let $w_r=H_\cA^{-1}P_ru_i\in\cE$, so that $\mathfrak h[w_r,\varphi]=\langle P_ru_i,\varphi\rangle$ for all $\varphi\in\cE$. Then $q_i=\sum_r\langle H_\cA^{-1}P_ru_i,P_ru_i\rangle=\sum_r\mathfrak h[w_r]$, and
\[
q_i
=\sum_r\langle P_r u_i, w_r\rangle
=\left\langle u_i,\sum_r P_r w_r\right\rangle
\leq\left\|\sum_rP_rw_r\right\|
\leq\left(\sum_{r,s}\|P_rw_s\|^2\right)^{1/2}=\sqrt{q_i}.
\]
Here the first equality holds because $\langle P_ru_i,w_r\rangle=\mathfrak h[w_r]$ is real and nonnegative, the second uses the symmetry of $P_r$ between elements of $\cE$, the first inequality is the Cauchy--Schwarz inequality with $\|u_i\|=1$, the second inequality is \eqref{eq:flatHodge} with the curl term discarded, and the last equality is $\sum_r\|P_rw_s\|^2=\mathfrak h[w_s]$. Hence $q_i\leq1$.
\end{proof}

With $(u_i)$, $e_i$, $\rho_{ij}$ and $d_{ij}$ as in the proof of Theorem~\ref{thm:ABmain}, put
\[
a_i:=\langle Ce_i,e_i\rangle_{\cE}
=\sum_j\left(\frac1{\lambda_i}+\frac1{\lambda_j}\right)d_{ij},
\]
where the second expression is \eqref{eq:diagonal}. Expanding $P_ru_i$ in the eigenbasis and using $[H_\cA,x_r]=-2\ir P_r$, that is, $\langle P_ru_i,u_j\rangle=(\ir/2)(\lambda_j-\lambda_i)\langle x_ru_i,u_j\rangle$, we find
\[
4q_i=\sum_j\frac{(\lambda_j-\lambda_i)^2}{\lambda_j}\rho_{ij}.
\]
Combining this with $\sum_jd_{ij}=2$, we obtain
\begin{equation}\label{eq:elementarypositivity}
a_i=\frac{4(1-q_i)}{\lambda_i}\geq0.
\end{equation}
This diagonal positivity is all that the eigenvalue argument requires. The identities \eqref{eq:ABmoments}--\eqref{eq:diagonal} were obtained without Lemma~\ref{lem:area} and remain available here. They imply
\[
\sum_j(\lambda_j-\Lambda)
\left(\frac1{\lambda_i}+\frac1{\lambda_j}\right)d_{ij}
=8-\Lambda a_i.
\]
Multiply by $(\Lambda-\lambda_i)_+$ and sum over $i$. On the left, the pairs $(i,j)$ with both eigenvalues below $\Lambda$ cancel by antisymmetry, and all remaining terms are nonnegative. It follows that $\sum_i(\Lambda-\lambda_i)_+(8-\Lambda a_i)\geq0$, meaning
\begin{equation}\label{eq:cutoff}
\sum_i\left(1-\frac{\lambda_i}{\Lambda}\right)_+a_i
\leq\frac8{\Lambda^2}\Tr(\Lambda-H_\cA)_+.
\end{equation}
Weyl's law gives $\Tr(\Lambda-H_\cA)_+\sim |\Om|\Lambda^2/(8\pi)$ as $\Lambda\to\infty$. Letting $\Lambda\to\infty$ in \eqref{eq:cutoff}, monotone convergence (the $a_i$ being nonnegative by \eqref{eq:elementarypositivity}) yields
\[
\sum_i a_i\leq\frac{|\Om|}{\pi}.
\]
We may therefore replace the last two inequalities in \eqref{eq:ABrefined} by
\[
\frac{|\Om|}{\pi}\geq\sum_i a_i\geq\sum_{i\leq k}a_i.
\]
The proof concludes as in Section~\ref{sec:AB}. In particular, this argument uses neither that $C$ is nonnegative nor that it is trace class.

\section{Proof of the bounded-field theorem}\label{sec:magnetic}

We retain the notation of Section~\ref{sec:AB}, with $\cA$ now real and $C^1$ on $\Om\setminus\cP$, $\cB=\curl\cA\in L^\infty(\Om\setminus\cP)$ and $b=\|\cB\|_\infty$ as in Theorem~\ref{thm:magnetic}.

\begin{proof}[Proof of Theorem~\ref{thm:magnetic}]
As in Section~\ref{sec:AB}, we may assume that $\Om$ has smooth boundary and that $\cA$ is $C^1$ on a neighbourhood of $\overline\Om$, and return to the passage to the limit at the end. For $w_1,w_2\in C_c^\infty(\Om)$, expand the two squared norms in the left-hand side of \eqref{eq:flatHodge} as in the proof of Lemma~\ref{lem:flatdiv}. The cross terms no longer cancel, but combine to $2\Re\langle[P_2,P_1]w_1,w_2\rangle$ with $[P_1,P_2]=\ir\cB$, whose absolute value is at most $2b\|w_1\|\|w_2\|\leq b(\|w_1\|^2+\|w_2\|^2)$. Discarding the curl term, we obtain
\begin{equation}\label{eq:magdiv}
\|P_1w_1+P_2w_2\|^2
\leq\sum_{r,s}\|P_rw_s\|^2+b\sum_s\|w_s\|^2,
\end{equation}
first on $C_c^\infty(\Om)$ and then, by density, on $\cE$. The right-hand side of \eqref{eq:magdiv} is $\sum_s\tilde{\mathfrak h}[w_s]$, where $\tilde{\mathfrak h}=\mathfrak h+b\|\cdot\|^2$ is the form of $\widetilde H:=H_\cA+b$. Let $(u_i)$ be an orthonormal eigenbasis of $H_\cA$, and put $\tilde\lambda_i=\lambda_i+b$. Taking $w_r=\widetilde H^{-1}P_ru_i$ in \eqref{eq:magdiv}, the argument in the proof of Lemma~\ref{lem:flatdiv} gives
\[
\tilde q_i:=\sum_r\|\widetilde H^{-1/2}P_ru_i\|^2\leq1.
\]

Define $\rho_{ij}$ and $d_{ij}$ from $(u_i)$ as in the proof of Theorem~\ref{thm:ABmain}. As before, the sum rules \eqref{eq:ABmoments} give
\[
\sum_jd_{ij}=2,\qquad
\sum_j(\tilde\lambda_i+\tilde\lambda_j)d_{ij}=8\tilde\lambda_i-4b.
\]
Therefore, as in \eqref{eq:elementarypositivity},
\[
\tilde a_i:=\sum_j\left(\frac1{\tilde\lambda_i}+\frac1{\tilde\lambda_j}\right)d_{ij}
=\frac{4(1-\tilde q_i)}{\tilde\lambda_i}\geq0.
\]
The cancellation argument leading to \eqref{eq:cutoff} now yields
\[
\Lambda\sum_i(\Lambda-\tilde\lambda_i)_+\tilde a_i
\leq\sum_i(\Lambda-\tilde\lambda_i)_+\left(8-\frac{4b}{\tilde\lambda_i}\right)
\leq8\Tr(\Lambda-\widetilde H)_+.
\]
Since Weyl's law for $\widetilde H$ is the same as for $H_\cA$, letting $\Lambda\to\infty$ gives
\begin{equation}\label{eq:magarea}
\sum_i \tilde a_i\leq\frac{|\Om|}{\pi}.
\end{equation}

Now fix $k\geq1$ and define $\mu_\pm$ as in the proof of Theorem~\ref{thm:ABmain}, with $\tilde\lambda$ in place of $\lambda$. For the pairs $i\leq k<j$, the weights $d_{ij}$ have total mass $2k$ and their $(\tilde\lambda_i+\tilde\lambda_j)$-weighted sum is $8km_k+4kb$, so that
\[
\mu_-+\mu_+=4m_k+2b,
\qquad \mu_-\leq\tilde\lambda_k\leq\mu_+.
\]
Repeating the proof of Theorem~\ref{thm:ABmain} with \eqref{eq:magarea} in place of Lemma~\ref{lem:area} and with $\Lambda=\tilde\lambda_k=\lambda_k+b$ gives \eqref{eq:magrefined}. Each denominator in \eqref{eq:magrefined} is at least $2\pi k/|\Om|$, so the inequality passes to the original domain by exhaustion, exactly as in the proof of Theorem~\ref{thm:ABmain}, since $\|\cB\|_{L^\infty(\Om_n)}\leq b$ for every subdomain. The bound on the denominators also gives $4m_k-\lambda_k+b>0$ in the limit.

Finally, since $m_k\leq\lambda_k$, we have $4m_k-\lambda_k+b\leq3\lambda_k+b$, and \eqref{eq:magrefined} implies $2k\bigl(1/(\lambda_k+b)+1/(3\lambda_k+b)\bigr)\leq|\Om|/\pi$, from which it follows that
\[
(\lambda_k+b)(3\lambda_k+b)\geq\frac{4\pi k}{|\Om|}(2\lambda_k+b).
\]
The smaller root of this quadratic in $\lambda_k$ is nonpositive; its larger root is the first expression in \eqref{eq:magindividual}. The second inequality in \eqref{eq:magindividual} follows from $\sqrt{\left(4\pi k/|\Om|\right)^2-4\pi kb/|\Om|+b^2}\geq4\pi k/|\Om|-b$.
\end{proof}

The dependence on the field in \eqref{eq:magindividual} cannot be omitted: constant-field examples \cite{FLW} show that a bound $\lambda_k\geq c\,4\pi k/|\Om|$ with $c$ independent of the field cannot hold with $c>1/2$. The first bound in \eqref{eq:magindividual} is stronger than $\lambda_k\geq2\pi k/|\Om|$ whenever $b<2\pi k(\sqrt2-1)/|\Om|$.

\begin{remark}\label{rem:admissible}
The hypotheses of Theorem~\ref{thm:magnetic} are, locally, no more general than ``an Aharonov--Bohm pole in a bounded background field.'' Indeed, on a disc about a pole $p$, solve $\Delta\psi=\cB$ and put $\cA^{\mathrm{reg}}=(-\partial_2\psi,\partial_1\psi)$. For bounded $\cB$, the potential $\cA^{\mathrm{reg}}$ is locally H\"older continuous, and $\cA-\cA^{\mathrm{reg}}$ is curl-free, hence, as in Remark~\ref{rem:gauge}, is equal to $\cA_{p,\alpha_p}$ up to a gauge transformation. Unbounded fields are not covered, even when the potential itself is bounded: for instance, $\cA=(x-p)^\perp/|x-p|$ satisfies $|\cA|=1$ but $\curl\cA=1/|x-p|$.
\end{remark}

\section{Further extensions}\label{sec:further}

We turn to the higher dimensional case. To this end, set
\[
\Lcln\gamma n:=\frac{\Gamma(\gamma+1)}{(4\pi)^{n/2}\Gamma(\gamma+1+n/2)},
\]
noting $\Lcln\gamma2=1/(4\pi(\gamma+1))$ and $\Lcln0n=\omega_n/(2\pi)^n$, where $\omega_n$ is the volume of the unit ball in $\R^n$. Higher-dimensional Aharonov--Bohm operators with singularities of codimension two have been studied, for example in \cite{FKLV}. Our results give Dirichlet spectral bounds for product configurations.

Let $D\subset\R^2$, $\cP\subset D$ and $\cA$ be as in Theorem~\ref{thm:ABmain}, and let $G\subset\R^{n-2}$, $n\geq3$, be a nonempty bounded open set. Let $H^{D\times G}_{\tilde\cA}$ be the self-adjoint operator in $L^2(D\times G)$ associated with the closure of the form
\[
u\longmapsto\int_{D\times G}
\left(|(-\ir\nabla_x-\cA(x))u|^2+|\nabla_yu|^2\right)\,\dr x\,\dr y,
\qquad u\in C_c^\infty((D\setminus\cP)\times G).
\]
This operator corresponds to the potential $\tilde\cA(x,y)=(\cA(x),0_{\R^{n-2}})$, whose singularities are contained in $\cP\times G$.

\begin{theorem}\label{thm:products}
Let $D$, $\cP$, $\cA$, $G$ and $H^{D\times G}_{\tilde\cA}$ be as above.
\begin{enumerate}
\item[(i)] The operator $H_{\tilde\cA}^{D\times G}$ satisfies the semiclassical bound
\[
\Tr\left(\Lambda-H_{\tilde\cA}^{D\times G}\right)_+^\gamma
\leq\Lcln\gamma n|D||G|\Lambda^{\gamma+n/2},
\qquad\gamma\geq1,\quad\Lambda\geq0.
\]
\item[(ii)] Suppose that $n\geq4$ and that $-\Delta_G$ satisfies P\'olya's inequality $\mathcal{N}_{-\Delta_G}(E)\leq\Lcln0{n-2}|G|E^{(n-2)/2}$ for $E\geq0$. Then $H_{\tilde\cA}^{D\times G}$ satisfies P\'olya's inequality
\[
\lambda_k\left(H_{\tilde\cA}^{D\times G}\right)\geq\lamP_k(D\times G),
\qquad
\lamP_k(D\times G):=4\pi^2\left(\frac{k}{\omega_n|D||G|}\right)^{2/n},
\qquad k\geq1.
\]
\end{enumerate}
\end{theorem}

\begin{proof}
Since the form is a sum of forms in the two variables,
\[
H_{\tilde\cA}^{D\times G}
=H_\cA^D\otimes I+I\otimes(-\Delta_G),
\]
where $-\Delta_G$ is the Dirichlet Laplacian on $G$, so that the eigenvalues of $H_{\tilde\cA}^{D\times G}$ are the sums $\lambda_i^D+\nu_j$ of the eigenvalues $(\lambda_i^D)$ of $H^D_\cA$ and $(\nu_j)$ of $-\Delta_G$. The product argument \cite{Laptev97} now applies: by Corollary~\ref{cor:ABriesz} and Berezin's bound for $-\Delta_G$ at order $\gamma+1$,
\[
\begin{split}
\Tr\left(\Lambda-H_{\tilde\cA}^{D\times G}\right)_+^\gamma
&=\sum_j\Tr\left(\Lambda-\nu_j-H_\cA^D\right)_+^\gamma
\leq\frac{|D|}{4\pi(\gamma+1)}\sum_j(\Lambda-\nu_j)_+^{\gamma+1}\\
&\leq\frac{\Lcln{\gamma+1}{n-2}}{4\pi(\gamma+1)}\,|D||G|\Lambda^{\gamma+n/2}
\end{split}
\]
for $\gamma\geq1$, and $\Lcln{\gamma+1}{n-2}/(4\pi(\gamma+1))=\Lcln\gamma n$. This proves (i).

The same argument gives (ii). Under its hypothesis, Corollary~\ref{cor:ABriesz} at order $(n-2)/2\geq1$ gives
\[
\begin{split}
\mathcal{N}_{H_{\tilde\cA}^{D\times G}}(\Lambda)
&=\sum_i\mathcal{N}_{-\Delta_G}(\Lambda-\lambda_i^D)
\leq\Lcln0{n-2}|G|\Tr(\Lambda-H_\cA^D)_+^{(n-2)/2}\\
&\leq\frac{\Lcln0{n-2}}{2\pi n}\,|D||G|\Lambda^{n/2}
=\Lcln0n|D||G|\Lambda^{n/2},
\end{split}
\]
which is (ii).
\end{proof}

\begin{theorem}\label{thm:fractional}
Let $\Om$, $\cP$ and $\cA$ be as in Theorem~\ref{thm:ABmain}.
\begin{enumerate}
\item[(i)] For $0<\gamma<1$ and $\Lambda\geq0$,
\[
\Tr\left(\Lambda-H_\cA\right)_+^\gamma
\leq\frac{\kappa_\gamma|\Om|}{4\pi(\gamma+1)}\Lambda^{\gamma+1},
\qquad
\kappa_\gamma:=(\gamma+1)\max_{0<\tau<3/2}
\tau\left(\frac{3-2\tau}{4-2\tau}\right)^\gamma.
\]
\item[(ii)] In the setting of Theorem~\ref{thm:products} with $n=3$, so that $G\subset\R$ is a bounded open set,
\[
\lambda_k\left(H_{\tilde\cA}^{D\times G}\right)
\geq\kappa_{1/2}^{-2/3}\lamP_k(D\times G)
\approx0.9345\,\lamP_k(D\times G),
\qquad k\geq1.
\]
\end{enumerate}
\end{theorem}

\begin{proof}
The bound (i) uses the freedom in the choice of $\Lambda$ in \eqref{eq:ABrefinedscalar}. For $\Lambda\leq\lambda_1$ there is nothing to prove, so let $\Lambda>\lambda_1$, choose $k$ with $\lambda_k<\Lambda\leq\lambda_{k+1}$, and put $\tau=4\pi k/(|\Om|\Lambda)$. Then \eqref{eq:ABrefinedscalar} can be rearranged as $(4\Lambda-8\pi k/|\Om|)m_k\geq\Lambda^2$, so that $m_k\geq\Lambda/(4-2\tau)$. Since $m_k<\Lambda$, this forces $\tau<3/2$. For $0<\gamma<1$, Jensen's inequality gives $\Tr(\Lambda-H_\cA)_+^\gamma\leq k(\Lambda-m_k)^\gamma$, and therefore, since $k=\tau\Lambda|\Om|/(4\pi)$, (i) follows. The maximum defining $\kappa_\gamma$ is attained at

\[
\tau_\gamma=\frac{7+\gamma-\sqrt{(7+\gamma)^2-48}}{4};
\]
for instance,
\[
\kappa_{1/2}=\sqrt{\frac{1863-297\sqrt{33}}{128}}\approx1.1070.
\]
By comparison, using only $m_k\geq2\pi k/|\Om|$, the same argument gives the constant $2/\sqrt3\approx1.1547$ in place of $\kappa_{1/2}$.

For (ii), the explicit eigenvalues of intervals give $\mathcal{N}_{-\Delta_G}(E)\leq|G|\sqrt E/\pi$, and hence
\[
\mathcal{N}_{H_{\tilde\cA}^{D\times G}}(\Lambda)
=\sum_i\mathcal{N}_{-\Delta_G}(\Lambda-\lambda_i^D)
\leq\frac{|G|}{\pi}\Tr(\Lambda-H_\cA^D)_+^{1/2}
\leq\kappa_{1/2}\frac{|D||G|}{6\pi^2}\Lambda^{3/2}.
\]
Since $\Lcln03=1/(6\pi^2)$, this yields (ii). By comparison, the factor $(3/4)^{1/3}\approx0.9086$ is obtained from \eqref{eq:sharpABsum} alone. Both bounds apply in particular to the Dirichlet Laplacian.
\end{proof}

\subsection*{AI usage statement}

The authors worked in collaboration with various AI tools after noticing that the methods described in \cite{Sp} could be applied to resolve the conjecture studied in this paper. These tools were also employed to verify and simplify proofs, search the literature, and then draft and edit parts of the manuscript. All arguments were carefully checked by the authors, who take full responsibility for the content.

\subsection*{Acknowledgements}

ML and IP have been partially supported by the AI for Math Fund, an initiative by Renaissance Philanthropy funded by XTX Markets. IP has been partially supported by NSERC and FRQNT.

\end{document}